\documentclass{article}

\def\cal{\mathcal}
\usepackage{color}
\usepackage{amsthm,amssymb,amsfonts,mathrsfs}
\usepackage{amsmath}
\catcode`@=11 \@addtoreset{equation}{section}

\catcode`@=12

\newtheorem{thm}{Theorem}[section]
\newtheorem{proposition}{Proposition}[section]
\newtheorem{lemma}{Lemma}[section]
\newtheorem{corollary}{Corollary}[section]
\theoremstyle{definition}

\newtheorem{definition}{Definition}[section]
\newtheorem{remark}{Remark}

\newtheorem{hypothesis}{Hypothesis}
\def\R{{\mathbb{R}}}

\newcommand{\cA}{{\mathcal A}}
\newcommand{\cY}{{\mathcal Y}}
\def\ds{\displaystyle}
\def\ve{\varepsilon}

\usepackage[colorlinks=true,urlcolor=red,linkcolor=blue,citecolor=red,bookmarks=true]{hyperref}

\title{Boundary controllability for degenerate/singular hyperbolic equations in nondivergence form with drift}
\author{
{\sc Genni Fragnelli$^a$, Dimitri Mugnai$^a$, Amine Sbai$^{b,c}$}\\\\
$^a$Dipartimento di Scienze Ecologiche e Biologiche,\\
Università della Tuscia,\\
Largo dell’Università, 01100 Viterbo, Italy,\\
email: (dimitri.mugnai, genni.fragnelli)@unitus.it\\\\
$^b$Hassan first University of Settat,\\
Faculty of Sciences and Technology, MISI Laboratory,\\
B.P. 577, Settat 26000, Morocco\\\\
$^c$Department of Applied Mathematics,\\
University of Granada, Granada, Spain.\\
email: a.sbai@uhp.ac.ma\\
}
\date{}
\begin{document}

\maketitle

\begin{abstract}
We study the null controllability for a degenerate/singular wave equation with drift in non divergence form. In particular, considering a control localized on the non degenerate boundary point,  we provide some conditions for the boundary controllability via energy methods and boundary observability.
\end{abstract}



\section{Introduction}
The aim of this paper is to study the null controllability of a degenerate hyperbolic equation with drift in presence of a singular term, with singularity at the same point where the leading coefficient degenerates.
To be more precise, we consider the following control problem:
\begin{align}\label{mainequation}
\begin{cases}u_{t t}-a(x) u_{x x}-\ds\frac{\lambda}{d(x)}u -b(x) u_x=0, & (t, x) \in Q_T, \\ u(t, 0)=0, \quad u(t, 1)=f(t), & t>0, \\ u(0, x)=u_0(x), \quad u_t(0, x)=u_1(x), & x \in(0,1),\end{cases}
\end{align}
where $Q_T=(0, T) \times(0,1)$, $\lambda \in \R$ and $u_0$, $u_1$ are the initial values. The control function $f$, which is used to steer the solution to its equilibrium state $0$ at a sufficiently large time $T>0$, acts on the non degenerate and non singular boundary point. The interest in this kind of equation comes from the
 standard linear theory for transverse waves in a string of length $L$  under tension $\mathcal T$. Actually, if $u(t,x)$ denotes the vertical displacement of the string from the $x$ axis at position $x\in (0,L)$ and time $t>0$, then the classical wave equation can be rewritten as
\[
\frac{\partial^2 u}{\partial t^2}(t,x)=a(x)\frac{\partial^2 u(t,x)}{\partial x^2}+b(x)\frac{\partial u}{\partial x}(t,x),
\]
where  $a(x):=\mathcal T(x)\rho^{-1}(x)$, $b(x):=\mathcal T'(x)\rho^{-1}(x)$. Here
$\rho(x)$ is the mass density of the string at position $x$, while $\mathcal T(x)$ denotes the tension in the string at position $x$ and time $t$. 
If the density is extremely large at some point, say $x=0$, then the previous equation {\sl degenerates} at $x=0$, in the sense that we can consider $a(0)=0$, while the remainder term is a drift one.
For this reason in \cite{BFM2022} the null controllability for \eqref{mainequation} with $\lambda =0$ is considered. For the same problem in divergence form we refer to \cite{gu} for the prototype case ($a(x)=x^\alpha,\; \alpha \in (0,1)$) and to \cite{alabau}) for a general function $a$.

Let us recall that  null controllability for the one dimensional {\it nondegenerate} wave equation can be attacked in several ways: for instance,  consider
\begin{equation}\label{classica}
\begin{cases}
u_{tt} - u_{xx}=f_\omega(t,x), & (t,x) \in Q_T,
\\
u(t, 0)= 0, \quad u(t,1)=f(t), & t \in (0,T),
\\
u(0,x)= u_0(x)\in H^1_0(0,1), \quad u_t(0,x)=u_1(x)\in L^2(0,1),& x \in (0, 1).
\end{cases}
\end{equation}
Here $u$ is the state, while $f_\omega$ and $f$ are the controls: one may have $f=0$ and $f_\omega$ acting as a control localized in the subset $\omega$ of $[0,1]$, or $f_\omega=0$ and $f$ acting as a boundary control. In any case, one looks for conditions in order to drive the solution to equilibrium at a given time $T$, i.e. given the initial data $(u_0, u_1)$ in a suitable space, we look for a control ($f$ or $f_\omega$) such that
\begin{equation}\label{NC1}
u(T,x)=u_t(T,x)=0, \quad \text{ for all } x \in (0,1). 
\end{equation}
Clearly, due to the finite speed of propagation of solutions of the wave equation, we cannot expect to have null controllability at any final time $T$ (as in the parabolic case), but we need $T$ to be sufficiently large: for equation \eqref{classica} it is well known that null controllability holds if $T>2$, see \cite[Chapter 4]{Russell}. Moreover, the Hilbert Uniqueness Method (HUM) permits to characterize such a control in terms of minimum of a certain functional. A related approach 
for \eqref{classica} with $f=0$ is showed in \cite{zuazua1}.
We also mention the recent paper \cite{s}, where the author studies \eqref{classica} with two linearly moving endpoints, establishing observability results in a sharp time and deriving exact boundary controllability results.

However, in
recent years  great attention is given to controllability issues for parabolic problems not only with a degenerate terms, but also with a singular term. Indeed,  many problems
coming from Physics and Biology (see \cite{AHSS2022}, \cite{KaZo}, \cite{vz}), Biology (see \cite{bf},  \cite{bfm}, \cite{epma}, \cite{f2018},\cite{f2020} and \cite{fy}) or Mathematical Finance
(see \cite{HW}) are described by degenerate parabolic equations with a singular term.
A common strategy in showing controllability is to prove global Carleman estimates  for
the operator which is the adjoint of the given one.
In this framework, new Carleman estimates and null controllability properties have been established in \cite{acf}, \cite{cmv2005}, \cite{cmv0}, \cite{fm2013} and \cite{mv} for regular degenerate coefficients, in \cite{bfm}, \cite{fm2016} and \cite{corrigendum} for non smooth degenerate coefficients and in \cite{fs}, \cite{fs1},  \cite{f2016}, \cite{fm2017}, \cite{fm2018}, \cite{fm2020} and \cite{v} for degenerate and singular coefficients.

Null controllability for wave equations with degeneracy and singularity in presence of pure powers and in divergence form has been recently tacled in \cite{ams}. 
As far as we know, this is the first paper to consider a {\sl degenerate} hyperbolic equation in {\sl non divergence form} with drift, where both the degeneracy and the singularity are described by more general functions. Clearly the presence of the drift term and of the singular term leads to use different spaces with respect to the ones in \cite{alabau}, in \cite{BFM2022} or  in \cite{ams} and give rise to some new difficulties.
As a consequence, the tools used in those papers cannot be simply adapted and a different functional setting is needed, together with suitable assumptions on $b$.
For this, following \cite{alabau} or \cite{BFM2022}, we consider two types of degeneracy for $a$ and $b$, which we introduce for a general function $g$.
\begin{definition}
A function $g:[0,1]\to \R$ is weakly degenerate at 0, (WD) for short, if $g$ $\in$ $C^0[0,1] \cap C^1(0,1]$ is such that $g(0)=0, g>0$ on $(0,1]$ and, if
\begin{align}\label{WD}
\sup _{x \in(0,1]} \frac{x\left|g^{\prime}(x)\right|}{g(x)}:=K_g,
\end{align}
then $K_g\in(0,1)$.
\end{definition}

\begin{definition}
 A function $g:[0,1]\to \R$ is strongly degenerate at 0, (SD) for short, if $g$ $\in$ $C^1[0,1]$ is such that $g(0)=0, g>0$ on $(0,1]$ and in \eqref{WD} we have $K_g \in[1,2)$.
\end{definition}

Roughly speaking, when $g(x)\sim x^K$, it is (WD) if $K\in (0,1)$ and (SD) if $K\in[1,2)$. Notice that the case $K\geq 2$ is not considered, since controllability cannot hold, not even when  $\lambda=0$, see \cite{BFM2022}.

\medskip

The paper is organized as follows: in Section \ref{section2} we study  the well posedness of the associated problem with Dirichlet boundary conditions, which we need to study the adjoint problem of \eqref{mainequation}; in Section \ref{section3} we consider such an adjoint problem  and we prove that for this kind of problem the associated (weighted) energy is constant in time. Moreover, we prove some estimates of the energy and, thanks to them, we  will prove in Section \ref{section4} boundary observabilities and null controllability for \eqref{mainequation}. The paper ends with the Appendix where we present some technical results, which we postpone for the readers' convenience.

\section{Well posedness for the problem with homogeneous Dirichlet boundary conditions}\label{section2}
Consider the degenerate/singular hyperbolic problem with Dirichlet boundary conditions
\begin{align}\label{homogequation}
\begin{cases}y_{t t}-a y_{x x}-\ds \frac{\lambda}{d}y-b y_x=0, & (t, x) \in(0,+\infty) \times(0,1) \\ y(t, 1)=y(t, 0)=0, & t \in(0,+\infty) \\ y(0, x)=y_T^0(x), & x \in(0,1) \\ y_t(0, x)=y_T^1(x), & x \in(0,1),\end{cases}
\end{align}
where the functional setting and the motivation for the notation of the initial data will be made clear below.

We assume the following hypothesis.
\begin{hypothesis}\label{hyp1}
The functions $a, b, d \in C^0[0,1]$ are such that 
\begin{enumerate}
\item $\ds \frac{b}{a} \in L^1(0,1)$,
\item $a(0)=d(0)=0$, $a,d>0$ on $(0,1]$,
\item there exist $K_1, K_2 \in(0,2)$ such that $K_1+ K_2 \le 2$ and  the functions
\begin{equation}\label{ipoa}
x \longmapsto \frac{x^{K_1}}{a(x)}
\end{equation}
and
\begin{equation}\label{ipod}
x \longmapsto \frac{x^{K_2}}{d(x)}
\end{equation}
are nondecreasing in a right neighborhood of $x=0$.
\end{enumerate}
\end{hypothesis}
It is clear that, if Hypothesis \ref{hyp1} holds, then
\begin{align}\label{limitxgamma/a}
\lim _{x \rightarrow 0} \frac{x^\gamma}{a(x)}=0
\end{align}
for all $\gamma>K_1$ and
\begin{align}\label{limitxgamma/da}
\lim _{x \rightarrow 0} \frac{x^\gamma}{d(x)}=0
\end{align}
for all $\gamma>K_2$.

Let us remark that if $a$ is (WD) or (SD), then \eqref{WD} implies that the function
\begin{align}\label{nondecreasxg/a}
x \mapsto \frac{x^\gamma}{a(x)} \text{ is nondecreasing in } (0,1] \text{ for all } \gamma \geq K_a.
\end{align}
In particular, the monotonicity condition of the map in \eqref{ipoa} holds  and
\begin{align}\label{boundxgamba}
\left|\frac{x^\gamma b(x)}{a(x)}\right| \leq \frac{1}{a(1)}\|b\|_{L^{\infty}(0,1)}
\end{align}
for all $\gamma \geq K_a$. For further purposes, let us introduce
\begin{align}\label{M}
M:=\frac{\|b\|_{L^{\infty}(0,1)}}{a(1)}, 
\end{align}
noticing that if $a$ and $b$ are both (WD) or (SD), Hypothesis \ref{hyp1} reduces to require $\frac{b}{a}\in L^1(0,1)$ and $K_a+K_d\leq2$.

Let us also remark that for the prototypes $a(x)=x^K$ and $b(x)=x^h$ with $K \in(0,1]$ and $h \geq 0$, the condition $\frac{b}{a} \in L^1(0,1)$ is clearly satisfied when $h>K-1$.
\medskip

In order to study the well-posedness of \eqref{homogequation}, we consider as in \cite{BFM2022} the well-known absolutely continuous weight function
\[
\eta(x):=\exp\left \{\int_{\frac{1}{2}}^x\frac{b(s)}{a(s)}ds\right \}
, \quad x\in [0,1],
\]
introduced by Feller  several years ago. Since $\ds \frac{b}{a}\in L^1(0,1)$, we find that $\eta\in C^0[0,1]\cap
C^1(0,1]$ is a strictly positive function. When $b$ degenerates at 0 not slower than $a$, for instance if $a(x)=x^K$ and $b(x)=x^h$,  $K\leq h$, then $\eta$ can be extended to a function of class $C^1[0,1]$. Now set
\[
\sigma(x):=a(x)\eta^{-1}(x),
\]
and introduce, as in \cite{cfr},
the following Hilbert spaces with the related inner products and norms
\[
 L^2_{\frac{1}{\sigma}}(0,1) :=\left\{ u \in L^2(0,1)\; \big|\; \|u\|_{ \frac{1}{\sigma}}<\infty \right\},
 \;  \langle u,v\rangle_{\frac{1}{\sigma}}:= \int_0^1u v\frac{1}{\sigma}dx, \; \|u\|^2_{\frac{1}{\sigma}} = \int_0^1 \frac{u^2}{\sigma}dx
\]
and
\[
H^1_{\frac{1}{\sigma}}(0,1) :=L^2_{\frac{1}{\sigma}}(0,1)\cap
H^1_0(0,1),
\]\[
\;  \langle u,v\rangle_{1,\frac{1}{\sigma}} :=   \langle u,v\rangle_{\frac{1}{\sigma}} + \int_0^1 \eta u'v'dx, \; \|u\|^2_{1,\frac{1}{\sigma}} = \|u\|^2_{\frac{1}{\sigma}} + \int_0^1\eta (u')^2 dx.
\]
Notice that the presence of the weight $\eta$ in the integrals above is non-essential, but it is useful to have easier calculations in some steps below.

The following result holds:
\begin{proposition}\label{propL2}
Assume Hypothesis $\ref{hyp1}$. If $u \in H^1_{\frac{1}{\sigma}}(0,1) $, then $ \ds\frac{u}{\sqrt{\sigma d}} \in L^2(0,1)$ and
there exists a positive constant $C>0$ such that
\begin{equation}\label{stima1CHP}
\int_0^1 \frac{u^2}{\sigma d }dx \le C \int_0^1\eta(u')^2dx.
\end{equation}
\end{proposition}
\begin{proof}
Taking $u \in H^1_{\frac{1}{\sigma}}(0,1)$, recalling the monotonicity condition in Hypothesis \ref{hyp1}, by Hardy's inequality one has
\[
\begin{aligned}
\int_0^1 \frac{u^2}{\sigma d }dx &\le \max_{[0,1]}\eta \int_0^1\frac{u^2}{ ad} dx \le  C\int_0^1\frac{u^2}{ x^{K_1+K_2}} dx \\
&\le  C\int_0^1\frac{u^2}{x^2} dx \le  \tilde C \int_0^1\eta (u')^2dx
\end{aligned}
\]
for some $C,\tilde C>0$.
\end{proof}
Let
\begin{equation}\label{CHP}
C_{HP}  \text{ be the best constant of \eqref{stima1CHP}},
\end{equation}
and let us remark that, if the monotonicity conditions in Hypothesis \ref{hyp1} are global, for instance if $a$ and $d$ are (WD) or (SD), then the estimate in the proof of Proposition \ref{propL2} shows that
\[
C_{HP}\leq \frac{4\max_{[0,1]}\eta}{a(1)d(1)\min_{[0,1]}\eta}.
\]

Another assumption we will use is the next one:
\begin{hypothesis}\label{hyp2}
The constant $\lambda \in \R$ is such that
\begin{equation}\label{lambda}
\lambda < \frac{1}{C_{HP}}.
\end{equation}
\end{hypothesis}

Under Hypotheses \ref{hyp1} and \ref{hyp2}, one can  consider on
$
H^1_{\frac{1}{\sigma}}(0,1) $ 
also the inner product
\[
\;  \langle u,v\rangle_1 :=   \langle u,v\rangle_{\frac{1}{\sigma}} + \int_0^1 \eta u'v'dx -\lambda \int_0^1 \frac{uv}{\sigma d}dx,
\]
which induces the norm
\[
\|u\|_1^2:= \|u\|^2_{\frac{1}{\sigma}} + \int_0^1\eta (u')^2 dx-\lambda \int_0^1 \frac{u^2}{\sigma d}dx.\]
For all $u \in H^1_{\frac{1}{\sigma}}(0,1)$, consider also the two norms
\[
\|u\|_{1, \circ}^2:= \int_0^1 (u')^2dx \quad \text{ and } \quad \|u\|_{1, \bullet}^2:= \int_0^1\eta (u')^2dx.
\]
By Proposition \ref{propL2} and the classical Poincar\'e inequality, it is straightforward to prove the equivalence below.
\begin{corollary}
\label{equivalenze}Assume Hypotheses $\ref{hyp1}$ and $\ref{hyp2}$. Then the norms
$
\|\cdot\|_{1,\frac{1}{\sigma}}^2$, $\|\cdot\|_1^2$,
$
\|\cdot\|_{1, \circ}^2$ and $\|\cdot\|_{1, \bullet}^2$
are equivalent in  $H^1_{\frac{1}{\sigma}}(0,1)$. 
\end{corollary}
Moreover, from \eqref{stima1CHP}, one can also deduce that there exists
 $C>0$ such that
 \begin{equation}\label{hpmon}
\int_0^1 \frac{v^2}{\sigma} dx
 \le C \int_0^1(v')^2dx\quad \forall \; v\in H^1_{\frac{1}{\sigma}}(0,1).
 \end{equation}
 Thus the
 spaces $H^1_0(0,1)$ and $H^1_{\frac{1}{\sigma}}(0,1)$ algebraically coincide. 
 
Now define the operator
\begin{equation}\label{defsigma}
Ay:=ay_{xx}+by_x = \sigma(\eta y_x)_x,
\end{equation}
for all $y \in D(A)$, where $D(A)$ is the Hilbert space
\[
H^2_{\frac{1}{\sigma}}(0,1) := \Big\{ u \in
H^1_{\frac{1}{\sigma}}(0,1)\; \big|\;Au \in
L^2_{\frac{1}{\sigma}}(0,1)\Big\},
\]
with inner products
\[
 \langle u,v\rangle_{2, \frac{1}{\sigma}} := \langle u,v\rangle_{1, \frac{1}{\sigma}}+  \langle Au,Av\rangle_{\frac{1}{\sigma}}\]
 or
\[
 \langle u,v\rangle_{2} := \langle u,v\rangle_1+  \langle Au,Av\rangle_{\frac{1}{\sigma}}.\]

The following integration by parts holds.
\begin{lemma}[see \cite{bfm}, Lemma 2.1]\label{Lemma2.1}
Assume Hypothesis $\ref{hyp1}$. If 
$u\in H^2_{{\frac{1}{\sigma}}}(0,1)$ and $v\in H^1_{{\frac{1}{\sigma}}}(0,1)$, then
\begin{equation}\label{Green}
 \langle Au,v\rangle_{\frac{1}{\sigma}}=-\int_0^1\eta u'v'dx.
\end{equation}
\end{lemma}


In order to study the well posedness of problem \eqref{homogequation}, we introduce the operator 
\[
A_\lambda y:= Ay + \frac{\lambda}{d}y, \quad \forall \; y \in D(A_\lambda),
\]
where
\begin{equation}\label{dal}
D(A_\lambda):=\left\{u \in
H^1_{\frac{1}{\sigma}}(0,1)\; \big|\;A_\lambda u \in
L^2_{\frac{1}{\sigma}}(0,1)\right\}.
\end{equation}

\begin{remark}\label{remdomini}
Notice that if $u\in H^1_{\frac{1}{\sigma}}(0,1)$ and $K_a+2K_d\leq 2$, reasoning as in the proof of Proposition \ref{propL2}, we immediately find that $\frac{u}{d}\in L^2_{\frac{1}{\sigma}}(0,1)$, and so $u \in D(A_\lambda)$ if and only if $u\in H^2_{\frac{1}{\sigma}}(0,1)$, that is $D(A_\lambda)=H^2_{\frac{1}{\sigma}}(0,1)$. 

If  $K_a+2K_d> 2$, the definition of $D(A_\lambda)$ is not sufficient to ensure that its members are smooth enough to justify the integration by parts given in \eqref{Green}: indeed, if $u\in D(A_\lambda)$, the request that $\sigma(\eta u')'+\lambda \frac{u}{d}\in L^2_{\frac{1}{\sigma}}$ doesn't permit to take $u$ in $H^2_{\frac{1}{\sigma}}$, since $\frac{u^2}{d^2\sigma}\sim \frac{u^2}{x^{2K_d+K_a}}$ near 0, and so, in general,  $\frac{u^2}{d^2\sigma}\not\in L^1(0,1)$.
\end{remark}

In view of the previous remark, we assume the following condition, which will assume also for the final controllability result, see Theorem \ref{thmNC}.
\begin{hypothesis}\label{hypvera}
Hypothesis \ref{hyp1} holds with $K_1+2K_2\leq 2$.
\end{hypothesis}
\begin{remark}
We notice that Hypothesis \ref{hypvera} excludes the case of $d$ being (SD). On the other hand, $a$ can be (SD), but in such a case $d$ must be (WD) with small $K_d$.
\end{remark}

Now, let us consider the matrix operator $\cA : D(\cA) \subset \mathcal H_0 \rightarrow \mathcal H_0$ given by
\begin{equation}\label{Acorsivo}
\cA:= \begin{pmatrix} 0 & Id\\
A_\lambda &0 \end{pmatrix}, \quad D(\cA):= D(A_\lambda) \times H^1_{{\frac{1}{\sigma}}}(0,1),
\end{equation}
and
\begin{equation}\label{H0corsivo}
\mathcal H_0:=H^1_{{\frac{1}{\sigma}}}(0,1)\times L^2_{{\frac{1}{\sigma}}}(0,1)
\end{equation}
is endowed with scalar product
\[
\langle \begin{pmatrix} u\\ v \end{pmatrix},\begin{pmatrix} w\\ z \end{pmatrix}\rangle_{\mathcal H_0}=\langle u,w\rangle_1+\langle v,z\rangle_{L^2_{{\frac{1}{\sigma}}}(0,1)}.
\]

In this way,  we can rewrite \eqref{homogequation} as the Cauchy problem
\begin{equation}\label{CP}
\begin{cases}
\dot \cY (t)= \cA \cY (t), & t \ge 0,\\
\cY(0) = \cY_0,
\end{cases}
\end{equation}
with
\[
\cY(t):= \begin{pmatrix} y\\ y_t \end{pmatrix} \; \text{ and }\; \cY_0:= \begin{pmatrix} y^0_T\\ y^1_T \end{pmatrix}.
\]
\begin{thm}\label{generator}
Assume Hypotheses $\ref{hyp2}$ and $\ref{hypvera}$. Then the operator $(\cA, D(\cA))$ is non positive with dense domain and generates a contraction semigroup  $(S(t))_{t \ge 0}$. 
\end{thm}
For the proof of the previous theorem we use the next result.
\begin{thm}[\cite{nagel}, Corollary 3.20]\label{densità}
Let $(\mathcal A,D(\mathcal A))$ be a dissipative operator on a reflexive Banach space such that $\mu I-\mathcal A$ is surjective for some $\mu >0$. Then $\mathcal A$ is densely defined and generates a contraction semigroup.
\end{thm}
\begin{proof}[Proof of Theorem $\ref{generator}$]
According to the previous theorem, it is sufficient to prove that $ \mathcal A:D(\mathcal A)\to \mathcal H_0$ is dissipative and that $I-\mathcal A$ is surjective.

\underline{$ \mathcal A$ is dissipative:} take $(u,v) \in D(\mathcal A)$. Then $(u,v) \in H^2_{{\frac{1}{\sigma}}}(0,1) \times H^1_{{\frac{1}{\sigma}}}(0,1)$  and so Lemma \ref{Lemma2.1} holds. Hence,
\[
\begin{aligned}
\langle \mathcal A (u,v), (u,v) \rangle_{\mathcal H_0} &=\langle (v, A_\lambda u), (u,v) \rangle _{\mathcal H_0} \\
&=\int_0^1 \eta u'v'dx - \lambda \int_0^1 \frac{u v}{\sigma d} dx+
 \int_0^1 vA_\lambda u\frac{1}{\sigma}dx  \\&
=\int_0^1\eta u'v'dx - \lambda \int_0^1 \frac{u v}{\sigma d} dx-\int_0^1\eta u'v'dx  + \lambda \int_0^1 \frac{u v}{\sigma d} dx =0.
\end{aligned}
\]
\underline{$I - \mathcal A$ is surjective:} 
take  $(f,g) \in \mathcal H_0=H^1_{{\frac{1}{\sigma}}}(0,1)\times L^2_{{\frac{1}{\sigma}}}(0,1)$. We have to prove that there exists $(u,v) \in D(\mathcal A)$ such that
\begin{equation}\label{4.3'}
 ( I-\mathcal A)\begin{pmatrix} u\\
v\end{pmatrix} = \begin{pmatrix}f\\
g \end{pmatrix} \Longleftrightarrow  \begin{cases} v= u -f,\\
-A_\lambda u + u= f+ g.\end{cases}
\end{equation}
Thus, define $F: H^1_{{\frac{1}{\sigma}}}(0,1) \rightarrow \R$ as
\[
F(z)=\int_0^1(f+g) z\frac{1}{\sigma}  dx .
\]
Obviously, $F\in H^{-1}_{{\frac{1}{\sigma}}}(0,1)$, the dual space of $H^1_{{\frac{1}{\sigma}}}(0,1)$ with respect to the pivot space $L^2_{{\frac{1}{\sigma}}}(0,1)$. Now, introduce the bilinear form $L:H^1_{{\frac{1}{\sigma}}}(0,1)\times H^1_{{\frac{1}{\sigma}}}(0,1)\to \R$ given by
\[
L(u,z):=  \int_0^1 u z \frac{1}{\sigma} dx + \int_0^1\eta u'z'dx -\lambda \int_0^1  \frac{u z}{\sigma d} dx
\]
for all $u, z \in H^1_{{\frac{1}{\sigma}}}(0,1)$, which is well defined by Proposition \ref{propL2}. Clearly, $L(u,z)$ is coercive: indeed, for all $u \in  H^1_{{\frac{1}{\sigma}}}(0,1)$, when $\lambda \le 0$ we have
\[
\begin{aligned}
L(u,u) &= \int_0^1  \frac{u^2}{\sigma} dx + \int_0^1\eta (u')^2dx -\lambda \int_0^1  \frac{u^2}{\sigma d} dx\\& \ge  \int_0^1  \frac{u^2}{\sigma} dx + \int_0^1\eta (u')^2dx = \|u\|^2_{1, \frac{1}{\sigma}}.
\end{aligned}
\]
On the other hand,   if $\lambda \in (0, C_{HP})$, then
\[
\begin{aligned}
L(u,u) \ge   \int_0^1  \frac{u^2}{\sigma} dx + (1-\lambda C_{HP})\int_0^1\eta (u')^2dx  \ge (1-\lambda C_{HP})\|u\|^2_{1, \frac{1}{\sigma}},
\end{aligned}
\]
by \eqref{stima1CHP}.
 Moreover $L(u,z)$ is 
 continuous: indeed, one has
\[
|L(u,z)| \le  \|u\|_{ L^2_{\frac{1}{\sigma}} (0,1) }\|z\|_ {L^2_{\frac{1}{\sigma}} (0,1) } +(\|\eta\|_{L^\infty(0,1)}+ |\lambda|)\|u'\|_{L^2(0,1)}\|z'\|_ {L^2(0,1)},
\]
 for all $u, z \in H^1_{{\frac{1}{\sigma}}}(0,1)$, and the conclusion follows by \eqref{hpmon}.

As a consequence, by the Lax-Milgram Theorem, there exists a unique solution $u \in H^1_{{\frac{1}{\sigma}}}(0,1)$ of
\[
L(u,z)= F(z)  \mbox{ for all }z\in H^1_{{\frac{1}{\sigma}}}(0,1),\]
namely
\begin{equation}\label{4.4}
\int_0^1  \frac{u z}{\sigma} dx + \int_0^1 \eta u'z'dx -\lambda \int_0^1  \frac{u z}{\sigma d} dx = \int_0^1(f+g) z \frac{1}{\sigma} dx 
\end{equation}
for all $z \in H^1_{{\frac{1}{\sigma}}}(0,1)$.

Now, take $v:= u-f$; then $v \in H^1_{{\frac{1}{\sigma}}}(0,1)$.
We will prove that $(u,v) \in D(\mathcal A)$ and solves \eqref{4.3'}. To begin with, \eqref{4.4} holds for every $z \in C_c^\infty(0,1).$ Thus we have
\[
\int_0^1 \eta u'z'dx = \int_0^1\left(f+g-u+ \lambda \frac{u}{d}\right) z \frac{1}{\sigma} dx 
\]
for every $z \in C_c^\infty(0,1).$ Hence $-(\eta u')'= \left(f+g-u+ \lambda \frac{u}{d}\right) \frac{1}{\sigma}$ in the distributional sense and a.e. in $(0,1)$; in particular, $-\sigma(\eta u')'=\left(f+g-u+ \lambda \frac{u}{d}\right) $ a.e. in $(0,1)$, that is $A_\lambda u=\sigma(\eta u')'+\lambda \frac{u}{d}=-\left(f+g-u\right) \in L^2_{ \frac{1}{\sigma}}(0,1)$; thus $u \in D(A)$ and \eqref{4.3'} holds.
\end{proof}

As usual in semigroup theory, the mild solution of \eqref{CP} obtained by  using Theorem \ref{generator} can be more regular: if $\cY_0 \in D(\mathcal A)$, then the solution is classical, in the sense that $\cY \in  C^1([0, +\infty); \mathcal H_0) \cap C([0, +\infty);D(\mathcal A))
$ and the equation in \eqref{homogequation} holds for all $t \ge0$. Hence, as in \cite[Corollary 4.2]{alabau} or in \cite[Proposition 3.15]{daprato}, one can deduce the next result.
\begin{thm}\label{thmmildclassolution}
Hypotheses  $\ref{hyp2}$ and $\ref{hypvera}$ hold. If $\left(y_T^0, y_T^1\right) \in H^1_{{\frac{1}{\sigma}}}(0,1)\times L^2_{{\frac{1}{\sigma}}}(0,1)$, then there exists a unique mild solution
$$
y \in C^1\left([0,+\infty) ; L_{\frac{1}{\sigma}}^2(0,1)\right) \cap C\left([0,+\infty) ; H_{\frac{1}{\sigma}}^1(0,1)\right)
$$
of \eqref{homogequation} which depends continuously on the initial data. Moreover, there exists $C>0$ such that
\begin{equation}\label{regolarita}
\int_0^1\frac{y_t(t,x)^2}{\sigma}dx+\int_0^1\eta y_x^2(t,x)dx\leq C\left( \int_0^1\frac{(y_T^1)^2}{\sigma}dx+\int_0^1\eta (y^0_T)_x^2dx\right)
\end{equation}
for every $t>0$. Moreover, if $\left(y_T^0, y_T^1\right) \in D(A_\lambda) \times H^1_{{\frac{1}{\sigma}}}(0,1)$, then the solution $y$ is classical, in the sense that
$$
y \in C^2\left([0,+\infty) ; L_{\frac{1}{\sigma}}^2(0,1)\right) \cap C^1\left([0,+\infty) ; H_{\frac{1}{\sigma}}^1(0,1)\right) \cap C\left([0,+\infty) ; D(A_\lambda)\right)
$$
and the equation in \eqref{homogequation} holds for all $t \geq 0$.
\end{thm}
\begin{proof}
We only prove \eqref{regolarita}, the rest being standard. First, assume that $y$ is a classical solution for which \eqref{Green} holds, for instance when $\left(y_T^0, y_T^1\right) \in C^\infty_c(0,1)\times C^\infty_c(0,1) \subset H^2_{{\frac{1}{\sigma}}}(0,1) \times H^1_{{\frac{1}{\sigma}}}(0,1)$. Then, taken any $t>0$, multiply the equation by $\frac{y_t}{\sigma}$ and integrate over $[0,1]\times [0,t]$, obtaining
\[
\begin{aligned}
\frac{1}{2}\int_0^t\int_0^1\frac{1}{\sigma}\frac{d}{d\tau}(y_{\tau}(\tau,x))^2dxd\tau&+\frac{1}{2}\int_0^t\int_0^1\eta\frac{d}{d\tau}(y_x(\tau,x)^2)dxd\tau\\&=\frac{\lambda}{2}\int_0^t\int_0^1\frac{1}{\sigma d}\frac{d}{d\tau}y^2(\tau,x)dxd\tau.
\end{aligned}
\]
As a consequence\footnote{Notice that this equality is an obvious consequence of Theorem \ref{thmref} below.},
\begin{equation}\label{intqual}
\begin{aligned}
\int_0^1\frac{y_t^2(t,x))}{\sigma}dx+\int_0^1\eta y_x^2(t,x)dx&=\int_0^1\frac{(y^1_T)^2}{\sigma}dx+\int_0^1\eta (y_T^0)_x^2dx\\&+\lambda \int_0^1\frac{y^2(t,x)}{\sigma d}dx-\lambda \int_0^1\frac{(y^0_T)^2}{\sigma d}dx.\end{aligned}
\end{equation}

Now, if $\lambda\leq 0$, the previous inequality, together with Proposition \ref{propL2}, immediately give \eqref{regolarita}. On the other hand, if $\lambda\in\left(0,\frac{1}{C_{HP}}\right)$, there exists $\ve>0$ such that $\lambda C_{HP}=1-\ve$. In this way, by Proposition \ref{propL2} we find
\[
\lambda \int_0^1\frac{y^2(t,x)}{\sigma d}dx\leq (1-\ve) \int_0^1\eta y_x^2dx,
\]
which, together with \eqref{intqual}, gives the desired statement.

Now, let $(y^T_0,y^T_1)\in H^1_{\frac{1}{\sigma}}(0,1)\times L^2_{\frac{1}{\sigma}}(0,1)$ and take let $(y_{0,n},y_{1,n})\in C^\infty_c(0,1)\times C^\infty_c(0,1)$ be such that $(y_{0,n},y_{1,n})\to (y^T_0,y^T_1)$ in $H^1_{\frac{1}{\sigma}}(0,1)\times L^2_{\frac{1}{\sigma}}(0,1)$ - recall that $C^\infty_c(0,1)$ is dense in $H^1_{{\frac{1}{\sigma}}}(0,1)$, and so in $L^2_{\frac{1}{\sigma}}(0,1)$, see \cite[Corollary 2.2]{cfrjee}. In this way, the mild solution $y_n$ to problem \eqref{homogequation} with $y^T_0$ replaced by $y_{0,n}$ and $y^T_1$ replaced by $y_{1,n}$ is also classical, and thus \eqref{regolarita} holds for $y_n$. Moreover, by linearity of the equation, it is readily seen that \eqref{regolarita} implies that $\frac{{(y_n)}_t}{\sqrt{\sigma}}+\sqrt{\eta}{(y_n)}_x$ is a Cauchy sequence in $L^2(0,1)$. This permits to pass to the limit and show that \eqref{regolarita} holds also for mild solutions.
\end{proof}

\begin{remark}\label{reversib_regul}
Due to the reversibility in time of the equation, solutions exist with the same regularity for $t<0$. Indeed, the associated matrix operator is skew-adjoint and it generates a $C_0$ group of unitary operators on the associated reference space $\mathcal{H}_0$ by Stone's Theorem. We will use this fact in the proof of the controllability result, by considering a backward problem whose final time data will be transformed in initial ones: this is the reason for the notation of the initial data in problem \eqref{homogequation}.
\end{remark}

\section{Energy estimates}\label{section3}
In this section we prove estimates of the energy associated to the solution of \eqref{homogequation} from below and from above. The former will be used in the next section to prove a
controllability result, while the latter is used here to prove a boundary observability
inequality. 

To begin with, we give the following lemma which will be crucial for the rest of the paper. Some points are similar to the ones given in \cite[Lemma 3.2]{BFM2022} or in \cite[Lemma 2.2]{fm}; however, due to the different domain of the operator $A$, some computations are more complicate that those in  \cite[Lemma 3.2]{BFM2022} or in \cite[Lemma 2.2]{fm}, for this reason we will give the proof of the next lemma in the Appendix.
\begin{lemma}\label{lemmalimits}Assume Hypothesis $\ref{hyp1}.$
\begin{enumerate}
\item If $y \in D(A)$ and $u \in H^1_{\frac{1}{\sigma}}(0,1)$, then $\ds
\lim_{x\rightarrow 0} u(x) y'(x)=0$.
\item If  $u \in H^1_{\frac{1}{\sigma}}(0,1)$, then   $\ds
\lim_{x\rightarrow 0} \frac{x}{a}u^2(x)=0.$
\item If  $u \in H^1_{\frac{1}{\sigma}}(0,1)$ and $K_1+ K_2 <2$, then   $\ds
\lim_{x\rightarrow 0} \frac{x}{ad}u^2(x)=0.$
\item If  $u \in H^1_{\frac{1}{\sigma}}(0,1)$, then   $\ds
\lim_{x\rightarrow 0} \frac{x^2}{ad}u^2(x)=0.$
\item  If $u\in D(A)$, then $\ds\lim_{x\rightarrow 0} x^2  (u'(x))^2=0$.
\item If  $u \in D(A)$, $K_1 \le 1$, then $\displaystyle\lim_{x\rightarrow 0} x  (u'(x))^2=0$.
\item If  $u \in D(A)$,  $K_1>1$ and $\ds\frac{xb}{a} \in L^\infty(0,1), $ then $\ds\lim_{x\rightarrow 0} x  (u'(x))^2=0$.
\end{enumerate}
\end{lemma}

Now, let us introduce the definition of the energy.
\begin{definition}
For a mild solution $y$ of \eqref{homogequation} we define its energy as the continuous function
\begin{align}\label{energy}
E_y(t)=\frac{1}{2} \int_0^1\left[\frac{1}{\sigma} y_t^2(t, x) + \eta y_x^2(t, x)-\frac{\lambda}{\sigma d}  y^2(t, x)\right]\,dx, \quad \forall t \geq 0
\end{align}
\end{definition}
The definition above guarantees that the classical conservation of the energy still holds true also in the degenerate/singular situation with a drift term, see Theorem \ref{thmref} below.

\begin{thm}\label{thmref}
Assume Hypothesis $\ref{hyp1}$ and let y be a mild solution of \eqref{homogequation}. Then
\begin{align}\label{waveref}
E_y(t)=E_y(0), \quad \forall \;t \geq 0 .
\end{align}
\end{thm}
\begin{proof}
First, suppose that $y$ is a classical solution. Then multiply the equation by $\ds\frac{y_t}{\sigma}$, integrate over $(0,1)$ and use the boundary conditions to get
\[
0 =\frac{1}{2} \int_0^1 \frac{d}{d t}\left(\frac{y_t^2}{\sigma}\right) d x-\left[\eta y_x y_t\right]_{x=0}^{x=1}+\int_0^1 \eta y_x y_{t x} d x -\int_0^1 \frac{\lambda}{\sigma d}  yy_t \, dx.
\]
Now, observe that, thanks to the boundary condition at $1$, $(\eta y_xy_t)(t,1)=0$. Moreover, thanks to Lemma \ref{lemmalimits}.1 applied with $u=y_t$, one has
\[
\lim_{x \rightarrow 0} \eta(x) y_t(t,x) y_x(t,x)=0.
\] 
Thus
\begin{align*}
0 &=\frac{1}{2} \int_0^1 \frac{d}{d t}\left(\frac{y_t^2}{\sigma}+\eta y_x^2\right) d x -\frac{1}{2}\int_0^1 \frac{d}{d t} \left(\frac{\lambda y^2}{\sigma d}\right)   \, dx\\
&=\frac{1}{2} \int_0^1 \frac{d}{d t}\left(\frac{y_t^2}{\sigma}+\eta y_x^2 -\frac{\lambda y^2}{\sigma d}\right) d x =\frac{1}{2} \frac{d}{d t} E_y(t).
\end{align*}
If y is a mild solution we can proceed with a standard approximation method with more regular initial data, as in the proof of Theorem \ref{thmmildclassolution}.
\end{proof}

Now we prove an inequality for the energy which we will use in the next section to prove the controllability result. For that, we start proving the following result.
\begin{thm}\label{thmuguaglianza}
Assume Hypothesis $\ref{hyp1}$ with $a$ and $d$ (WD) or (SD) such that $K_a+2K_d\leq 2$. If y is a classical solution of \eqref{homogequation}, then for any $T>0$ we have
\begin{equation}\label{uguaglianza}
\begin{aligned}
& \frac{1}{2} \eta(1) \int_0^T y_x^2(t, 1) d t=\int_0^1\left[\frac{x^2 y_x y_t}{\sigma}\right]_{t=0}^{t=T} d x+\frac{1}{2} \int_{Q_T} \left(2- \frac{x(a'-b)}{a} \right)x\frac{y_t^2}{\sigma} d x d t \\
&+
\int_{Q_T} x\eta  y_x ^2 dxdt  
- \frac{1}{2} \int_{Q_T} x^2 \eta y_x^2\frac{b}{a} dxdt+  \frac{ \lambda}{2}\int_{Q_T} \frac{y^2x}{\sigma d} \left( 2- \frac{x(a'-b)}{a} - \frac{x d'}{d}\right) dxdt.
\end{aligned}
\end{equation}
As a consequence,  if Hypothesis $\ref{hyp2bis}$ holds and if $y$ is a mild solution of \eqref{homogequation}, then $y_x(\cdot,1)\in L^2(0,T)$ for every $T>0$ and there exist two constants $C_i>0$, $i=1,2$, such that
\begin{align}\label{ineqetayx2}
\eta(1) \int_0^T y_x^2(t, 1) d t \leq C_1E_y(0)+C_2T  E_y(0).
\end{align}
\end{thm}
\begin{proof}First, assume that $y$ is a classical solution of \eqref{homogequation}. Then, multipling the equation in \eqref{homogequation} by $\ds\frac{x^2 y_x}{\sigma}$ and integrating by parts over $Q_T$ we have
\begin{equation}\label{eq1}
\begin{split}
0 &=\int_0^1\left[\frac{x^2 y_x y_t}{\sigma}\right]_{t=0}^{t=T} d x-\frac{1}{2} \int_0^T\left[\frac{x^2}{\sigma} y_t^2\right]_{x=0}^{x=1} d t+\frac{1}{2} \int_{Q_T}\left(\frac{x^2}{\sigma}\right)_x y_t^2 d x d t\\
& - \int_{Q_T} Ay \frac{x^2y_x}{\sigma}dx- \lambda\int_0^1 \frac{x^2 y y_x}{\sigma d} dx .
\end{split}
\end{equation}
Observe that by assumptions $Au \in L^2_{\frac{1}{\sigma}}(0,1)$, $y_x \in L^2(0,1)$ and $\ds \frac{x^2}{a}$ is bounded, thus $\ds \int_{Q_T} Ay \frac{x^2y_x}{\sigma}dx$ is well defined by H\"older's inequality. Moreover, also
\[
\int_0^1 \frac{x^2y y_x}{\sigma d} dx \in \R.
\]
Indeed, by Proposition \ref{propL2} there exists $c>0$ such that
\[
\left|\int_0^1 \frac{x^2 y y_x}{\sigma d} dx\right| \le \max_{[0,1] }\sqrt{\eta}\int_0^1\frac{|y|}{\sqrt{\sigma d}}| y_x |\frac{x^{K_1+ K_2} }{\sqrt{a d}}dx \le c\int_0^1\frac{|y|}{\sqrt{\sigma d}} |y_x| dx\footnote{Let us remark that, if $a$ and $d$ are (WD) or (SD), then a good constant in such an inequality is $\frac{\max_{[0,1]}\sqrt{\eta}}{\sqrt{a(1)d(1)}}$.},
\]
and H\"older's inequality permits to conclude. Hence, by \eqref{eq1},
\begin{equation}\label{eq2}
\begin{aligned}
0 & =\int_0^1\left[\frac{x^2 y_x y_t}{\sigma}\right]_{t=0}^{t=T} d x-\frac{1}{2} \int_0^T\left[\frac{x^2}{\sigma} y_t^2\right]_{x=0}^{x=1} d t+\frac{1}{2} \int_{Q_T}\left(\frac{x^2}{\sigma}\right)_x y_t^2 d x d t \\
&- \int_{Q_T}  (\eta y_x)_x x^2y_x dxdt - \lambda\int_{Q_T}\frac{x^2 y y_x}{\sigma d} dxdt \\
&=\int_0^1\left[\frac{x^2 y_x y_t}{\sigma}\right]_{t=0}^{t=T} d x-\frac{1}{2} \int_0^T\left[\frac{x^2}{\sigma} y_t^2\right]_{x=0}^{x=1} d t+\frac{1}{2} \int_{Q_T} \left(2- \frac{x(a'-b)}{a} \right)x\frac{y_t^2}{\sigma} d x d t \\
&- \int_0^T  
\left[x^2 \eta y_x^2\right]_{x=0}^{x=1} d t + \int_{Q_T} \eta y_x( x^2y_x)_x dxdt
- \frac{ \lambda}{2}\int_{Q_T}\frac{x^2}{\sigma d} (y^2)_x dxdt \\
&=\int_0^1\left[\frac{x^2 y_x y_t}{\sigma}\right]_{t=0}^{t=T} d x-\frac{1}{2} \int_0^T\left[\frac{x^2}{\sigma} y_t^2\right]_{x=0}^{x=1} d t+\frac{1}{2} \int_{Q_T} \left(2- \frac{x(a'-b)}{a} \right)x\frac{y_t^2}{\sigma} d x d t \\
& -\int_0^T  
\left[x^2 \eta y_x^2\right]_{x=0}^{x=1} d t + \int_{Q_T} \eta y_x(2x y_x + x^2 y_{xx}) dxdt
- \frac{ \lambda}{2}\int_0^T\left[\frac{x^2y^2}{\sigma d}\right]_{x=0}^{x=1} dt \\
&+ \frac{ \lambda}{2}\int_{Q_T}\left(\frac{x^2}{\sigma d}\right)_x y^2dxdt\\
& = \int_0^1\left[\frac{x^2 y_x y_t}{\sigma}\right]_{t=0}^{t=T} d x-\frac{1}{2} \int_0^T\left[\frac{x^2}{\sigma} y_t^2\right]_{x=0}^{x=1} d t -\frac{1}{2}\int_0^T  
\left[x^2 \eta y_x^2\right]_{x=0}^{x=1} d t-\\&
 \frac{ \lambda}{2}\int_0^T\left[\frac{x^2y^2}{\sigma d}\right]_{x=0}^{x=1} dt+\frac{1}{2} \int_{Q_T} \left(2- \frac{x(a'-b)}{a} \right)x\frac{y_t^2}{\sigma} d x d t +
\int_{Q_T} x\eta y_x^2 dxdt \\
& - \frac{1}{2} \int_{Q_T} x^2 \eta y_x^2\frac{b}{a} dxdt+  \frac{ \lambda}{2}\int_{Q_T} \frac{y^2x}{\sigma d} \left( 2- \frac{x(a'-b)}{a} - \frac{x d'}{d}\right) dxdt.
\end{aligned}
\end{equation}

As in \cite{BFM2022}, using Lemma \ref{lemmalimits}, we can prove 
\begin{align*}
\int_0^T\left[\frac{x^2}{\sigma} y_t^2\right]_{x=0}^{x=1} d t=0;
\end{align*}
in addition, thanks to Lemma \ref{lemmalimits}, $\lim_{x\rightarrow 0} x^2y_x^2 (t,x)=0$ and $\ds\lim_{x\rightarrow 0} y^2(t,x) \frac{x^2}{a  d}=0$, so that
\[
\int_0^T\left[x^2 \eta y_x^2\right]_{x=0}^{x=1} d t=\eta(1)\int_0^T  y_x^2(t, 1) d t
\]
and
$$
 \lambda \int_0^T\left[y^2 \frac{x^2}{\sigma  d}\right]_{x=0}^{x=1}=0,
$$
so that \eqref{uguaglianza} holds true.

Now assume Hypothesis \ref{hyp2bis} and let us prove that inequality \eqref{ineqetayx2} holds. First, suppose that $y$ is a classical solution. As in \cite[Theorem 3.4]{BFM2022}, one can prove that
\begin{align*}
\int_0^1\left[\frac{x^2 y_x y_t}{\sigma}\right]_{t=0}^{t=T} d x \leq 2 \max \left\{\frac{1}{a(1)}, 1\right\} E_y(0) 
\end{align*}

Then, we estimate the distributed terms in the previous equation.
Using \eqref{nondecreasxg/a}, one has
$$
\frac{1}{2}\left|\int_{Q_T} \eta \frac{b x^2}{a} y_x^2 d x d t\right| \leq \frac{M}{2} \int_{Q_T} \eta y_x^2 d x d t,
$$
where $M$ is defined in \eqref{M}; moreover
$$
\left|\frac{1}{2} \int_{Q_T} \frac{x y_t^2}{\sigma}\left(2-\frac{x\left(a^{\prime}-b\right)}{a}\right) d x d t\right| \leq \frac{2+K_a+M}{2} \int_{Q_T} \frac{y_t^2}{\sigma} d x d t
$$
and, by Proposition \ref{propL2},
\[
\begin{aligned}
\left\vert \frac{ \lambda}{2}\int_{Q_T} \frac{y^2x}{\sigma d} \left( 2- \frac{x(a'-b)}{a} - \frac{x d'}{d}\right) dxdt \right\vert &\leq \frac{| \lambda|}{2}\int_{Q_T} \frac{y^2}{\sigma d} (2+K_a+K_d+M)\,dx dt\\
& \le \frac{| \lambda|(2+K_a+K_d+M)C_{HP}}{2}\int_{Q_T} \eta y_x^2dxdt.
\end{aligned}
\]
Hence, by \eqref{uguaglianza} and the previous computations, thanks to Theorem \ref{thmref}, we get
\[
\begin{aligned}
\frac{1}{2} \eta(1)  \int_0^Ty_x^2(t, 1) d t & \leq2 \max \left\{\frac{1}{a(1)}, 1\right\} E_y(0)+\frac{\left(2+K_a+M\right)}{2} \int_{Q_T} \frac{y_t^2}{\sigma} d x d t\\
& +\left(\frac{| \lambda|(2+K_a+K_d+M)C_{HP}}{2}+\frac{M}{2}+1\right) \int_{Q_T} \eta y_x^2 d x d t\\
& \leq C_1 E_y(0)+C_2TE_y(0)
\end{aligned}
\]
where
\[
C_1=2 \max \left\{\frac{1}{a(1)}, 1\right\}
\]
and
\[
C_2=\max\left\{ | \lambda|(2+K_a+K_d+M)C_{HP}+M+2,2+K_a+M\right\}.
\]

Now, let $y$ be the mild solution associated to the initial data $\left(y_0, y_1\right) \in \mathcal{H}_0$, see \eqref{H0corsivo}, and consider a sequence $\left\{\left(y_0^n, y_1^n\right)\right\}_{n \in \mathbb{N}} \subset D(\mathcal{A})=H^2_{\frac{1}{\sigma}}(0,1)\times H^1_{\frac{1}{\sigma}}(0,1)$ - see \eqref{Acorsivo} - that approximates $\left(y_0, y_1\right)$ and let $y^n$ be the classical solution of \eqref{homogequation} associated to $\left(y_0^n, y_1^n\right)$. Clearly $y^n$ satisfies \eqref{ineqetayx2} and by linearity $({y_n})_x$ is a Cauchy sequence in $L^2(0,T)$ and we can conclude by passing to the limit.
\end{proof}

We shall also use an estimate from above for the energy. 
The next preliminary result holds.
\begin{thm}\label{thmuguaglianza2}
Assume Hypothesis $\ref{hyp1}$ with $a$ and $d$ (WD) or (SD) such that $K_a+2K_d\leq 2$. If $y$ is a classical solution of \eqref{homogequation}, then for any $T>0$ the following identity holds:
\begin{align}\label{uguaglianza2}
&\frac{1}{2} \eta(1) \int_0^T y_x^2(t, 1) d t=\int_0^1\left[\frac{x y_x y_t}{\sigma}\right]_{t=0}^{t=T} d x-\frac{1}{2} \int_{Q_T} x \eta \frac{b}{a} y_x^2 d x d t+\frac{1}{2} \int_{Q_T} \eta y_x^2 d x d t \notag\\
& +\frac{1}{2} \int_{Q_T}\left(1-\frac{x\left(a^{\prime}-b\right)}{a}\right) \frac{1}{\sigma} y_t^2 d x d t +\frac{\lambda}{2} \int_{Q_T} \frac{y^2}{\sigma d}\left(1-x\left(\frac{d^{\prime}}{d}+\frac{a^\prime-b}{a}\right)\right).
\end{align}
\end{thm}
\begin{proof}
Multiplying the equation in \eqref{homogequation} by $\ds\frac{x y_x}{\sigma}$, integrating by parts over $Q_T$ and recalling
\eqref{defsigma} we have, as in \cite[Theorem 3.6]{BFM2022},
\begin{equation}\label{eq1_1}
\begin{aligned}
\frac{1}{2}\int_0^T \left[x\eta y_x^2\right]_{x=0}^{x=1}dt &= \int_0^1\left[\frac{xy_xy_t}{\sigma}\right]_{t=0}^{t=T} dx   - \frac{1}{2} \int_{Q_T} x \eta \frac{b}{a}  y_x^2 dxdt \\
&+\frac{1}{2} \int_{Q_T}\left(1-\frac{x(a'-b)}{a}\right) \frac{1}{\sigma}y_t^2 dxdt+ \frac{1}{2} \int_{Q_T} \eta y_x^2 dxdt\\
&-\lambda\int_{Q_T} \frac{x yy_x}{d \sigma}dxdt.
\end{aligned}
\end{equation}
Now, consider the last term in the previous equality:
\begin{equation}\label{eq1_2}
\begin{aligned}
-\lambda\int_{Q_T} \frac{x yy_x}{d \sigma}dxdt&= -\frac{\lambda}{2} \int_{Q_T} \frac{x\left(y^2\right)_x}{\sigma d}dx dt\\
&=-\frac{\lambda}{2} \int_0^{T}\left[\frac{x y^2}{\sigma d}\right]_{x=0}^{x=1}+\frac{\lambda}{2} \int_{Q_T} \frac{y^2}{\sigma d}\left(1-x \frac{\sigma d^{\prime}+\sigma^{\prime} d}{\sigma d}\right)dx dt \\
&=-\frac{\lambda}{2} \int_0^{T}\left[\frac{x y^2}{\sigma d}\right]_{x=0}^{x=1}+\frac{\lambda}{2} \int_{Q_T} \frac{y^2}{\sigma d}\left[1-x\left(\frac{d^{\prime}}{d}+\frac{a^{\prime}-b}{a}\right)\right]dx dt.
\end{aligned}
\end{equation}
Clearly $\ds\left(\frac{x y^2}{\sigma d}\right)(t,1)=0$. Moreover, the fact that $K_a+ 2K_d \leq 2$ implies that $K_a+ K_d <2$, so by Lemma \ref{lemmalimits}.3 we find
\[\lim_{x \rightarrow 0 }\left(\frac{x y^2}{\sigma d}\right)(t,x)=0\]
and the conclusion follows.
\end{proof}

In order to conclude the estimate which leads to the observability inequality, we distinguish between two cases; the first case is when $a$ is (WD) or (SD) with $K_a = 1$, while the second one is for $a$ (SD) and $K_a \neq 1.$

In order to face the original problem, we introduce the next assumption, which collects all the needed requirements, in particular Hypothesis \ref{hypvera}.
\begin{hypothesis}\label{hyp2bis}
Hypothesis \ref{hyp2} holds. Moreover, functions $a, b, d \in C^0[0,1]$ are such that
\begin{itemize}
\item $\dfrac{b}{a}\in L^1(0,1)$;
\item $a$ and $d$ are (WD) or (SD) with $K_a+2K_d\leq 2$.
\end{itemize}
\end{hypothesis}

\begin{hypothesis}\label{hyp3}
Hypothesis \ref{hyp2bis} holds with $K_a\leq 1$ and $K_a<2-2M$.
\end{hypothesis}
\begin{remark}\label{Rem9}
\begin{enumerate}
\item The assumption $K_a <2-2M$ implies that $M<1$.
\item The assumption $K_a <2-2M$ is clearly satisfied if $\ds \|b\|_{L^\infty(0,1)}< \frac{a(1)}{2}$. Indeed in this case $2-2M >1$ and $K_a \le 1$.
\item The assumption $K_a <2-2M$ is not surprising since it is the same assumption made in \cite{BFM2022} when $\lambda =0$.
\end{enumerate}\end{remark}

The observability inequality we get in this case is the following:
\begin{thm} \label{stima1} Assume Hypothesis $\ref{hyp3}$. If $y$ is a mild solution of \eqref{homogequation}, then there exists a positive constant $C_3$ such that for any $T>0$ we have
$$
\begin{aligned}
\eta(1) \int_0^T y_x^2(t, 1) d t &\geq 2T\left(1- \frac{K_a}{2}-M -|\lambda| C_{HP}(2+K_a+K_d)\right)E_y(0)\\
&-\left\{4 \max\left\{1, \frac{1}{a(1)}\right\}+K_a\max \left\{1, C_{HP}\max_{[0,1]}d\right\}\right\}E_y(0)\\
\end{aligned}
$$
if $\lambda \le 0$, and
\[
\begin{aligned}
\eta(1)  \int_0^T y_x^2(t, 1) d t&  \ge  2T\left(1- \frac{K_a}{2}-M\right)E_y(0)+\lambda (2-K_a-K_d-M)\int_{Q_T} \frac{y^2}{\sigma d}dxdt\\
&- \frac{1}{\epsilon} \left( 4\max\left\{\frac{1}{a(1)}, 1\right\}+K_a\max \left\{1, C_{HP} \max_{[0,1]}d\right\}\right)E_y(0)
\end{aligned}
\]
where $\epsilon \in (0,1)$ is such that
\begin{equation}\label{epsilon}
\lambda = \frac{1}{C_{HP}} -\frac{\epsilon}{C_{HP}}>0.
\end{equation}
\end{thm}
\begin{proof}
As a first step, assume that $y$ is a classical solution. By multiplying  the equation in \eqref{homogequation} by $\ds \frac{-K_a y}{2 \sigma}$ and integrating by parts on $Q_T$, we have 
$$
0=-\frac{K_a}{2} \int_0^1\left[\frac{y y_t}{\sigma}\right]_{t=0}^{t=T} d x+\frac{K_a}{2} \int_{Q_T} \frac{y_t^2}{\sigma} d x d t-\frac{K_a}{2} \int_{Q_T} \eta y_x^2 d x d t +\frac{\lambda K_a}{2} \int_{Q_T} \frac{y^2}{\sigma d}dx dt.
$$
Summing the previous equality to \eqref{uguaglianza2} multiplied by $2$, we have
$$\begin{aligned}
\eta(1) & \int_0^T y_x^2(t, 1) d t=2 \int_0^1\left[\frac{x y_x y_t}{\sigma}\right]_{t=0}^{t=T} d x-\frac{K_a}{2} \int_0^1\left[\frac{y y_t}{\sigma}\right]_{t=0}^{t=T} d x+\frac{K_a}{2} \int_{Q_T} \frac{y_t^2}{\sigma} d x d t \\
& -\frac{K_a}{2} \int_{Q_T} \eta y_x^2 d x d t+\int_{Q_T} \frac{y_t^2}{\sigma} d x d t+\int_{Q_T} \eta y_x^2 d x d t-\int_{Q_T} x \eta \frac{b}{a} y_x^2 d x d t \\
& -\int_{Q_T} \frac{x\left(a^{\prime}-b\right)}{a} \frac{y_t^2}{\sigma} d x d t +\lambda \int_{Q_T} \frac{y^2}{\sigma d}\left(1-x\left(\frac{d^{\prime}}{d}+\frac{a^\prime-b}{a}\right)\right) +\frac{\lambda K_a}{2} \int_{Q_T} \frac{y^2}{\sigma d}dx dt. 
\end{aligned}$$
Analogously to what done in the proof of \cite[Theorem 3.7]{BFM2022} for $\lambda=0$, we find
\begin{equation}\label{puntpar}
\begin{aligned}
\eta(1)  \int_0^T y_x^2(t, 1) d t&\ge 2 \int_0^1 \left[ \frac{xy_xy_t}{\sigma}\right]_{t=0}^{t=T} dx- \frac{K_a}{2} \int_0^1 \left[ \frac{y_t y}{\sigma}\right]_{t=0}^{t=T}dx 
\\&
+ \left(1- \frac{K_a}{2}-M\right)\int_{Q_T} \eta y_x^2 dxdt\\&+ \left(1- \frac{K_a}{2}-M\right)\int_{Q_T}\frac{y_t^2}{\sigma} dxdt \\& +\lambda \int_{Q_T} \frac{y^2}{\sigma d}\left(1+\frac{K_a}{2}-x\left(\frac{d'}{d}+\frac{a'-b}{a}\right)\right)dxdt.
\end{aligned}
\end{equation}

First, we concentrate on the boundary terms. By the Schwartz inequality, for all $\tau \in [0, T]$ we have
\begin{equation}\label{bordo1}
\begin{aligned}
\left| 2\int_0^1\frac{xy_x(\tau,x) y_t(\tau,x)}{\sigma(x)} dx \right|& \le  \int_0^1\frac{x^2}{\sigma(x)}y_x^2(\tau,x)dx+ \int_0^1\frac{y_t^2(\tau,x)}{\sigma(x)}dx \\
&\le \frac{1}{a(1)} \int_0^1(\eta y_x^2)(\tau,x)dx+ \int_0^1\frac{y_t^2(\tau,x)}{\sigma (x)}dx\\
&\le \max\left\{\frac{1}{a(1)}, 1\right\}\left( \int_0^1\eta y_x^2(\tau,x)dx+\int_0^1\frac{y_t^2(\tau,x)}{\sigma} \right),
\end{aligned}
\end{equation}
and by Proposition \ref{propL2}
\begin{equation}\label{bordo2}
\begin{aligned}
\frac{K_a}{2} \left|\int_0^1 \frac{yy_t}{\sigma}(\tau,x)dx\right|&\le\frac{K_a}{4} \int_0^1\frac{y_t^2}{\sigma}(\tau,x)dx +\frac{K_aC_{HP}}{4} \max_{[0,1]}d\int_0^1\eta y_x^2(\tau,x)dx\\
& \le \frac{K_a}{4}\max \left\{1, C_{HP}\max_{[0,1]}d\right\}\left( \int_0^1\eta y_x^2(\tau,x)dx+\int_0^1\frac{y_t^2(\tau,x)}{\sigma} \right)
\end{aligned}
\end{equation}
for all $\tau \in [0, T]$.

Now, we proceed in different ways, according to the sign of $\lambda$.

\noindent{{\sl Case  $\lambda\in\left(0,\frac{1}{C_{HP}}\right)$}}. In this case, by \eqref{epsilon} and by \eqref{stima1CHP},
\[
\int_0^1 \eta y_x^2dx - \lambda \int_0^1 \frac{y^2}{\sigma d} dx \ge \int_0^1 \eta y_x^2dx -(1-\epsilon) \int_0^1 \eta y_x^2dx = \epsilon \int_0^1 \eta y_x^2dx.
\]
Thus, by Theorem \ref{thmref}, for all $\tau \in [0, T]$
\[
\begin{aligned}
2E_y(0)&=2E_y(\tau) \ge \epsilon \int_0^1 (\eta y_x^2)(\tau, x)dx + \int_0^1 \frac{y_t^2(\tau, x)}{\sigma(x)}dx \\&\ge \epsilon \left(\int_0^1 (\eta y_x^2)(\tau, x)dx + \int_0^1 \frac{y_t^2(\tau, x)}{\sigma(x)}dx\right)
\end{aligned}
\]
so that
\begin{equation}\label{Importante}
\int_0^1 (\eta y_x^2)(\tau, x)dx + \int_0^1 \frac{y_t^2(\tau, x)}{\sigma(x)}dx \le \frac{2}{\epsilon} E_y(0)
\end{equation}
for all $\tau \in [0, T]$.

Thus, by \eqref{bordo1} and \eqref{Importante} we get for any $\tau\in [0,T]$
\begin{equation}\label{wer}
2\left| \int_0^1\frac{xy_x(\tau,x) y_t(\tau,x)}{\sigma(x)} dx \right| \le \frac{2}{\epsilon}\max\left\{\frac{1}{a(1)}, 1\right\}E_y(0),
\end{equation}
while from \eqref{bordo2} and \eqref{Importante}
\begin{equation}\label{wer2}
\frac{K_a}{2} \left|\int_0^1 \frac{yy_t}{\sigma}(\tau,x)dx\right|\le \frac{K_a}{2\epsilon}\max \left\{1, C_{HP}\max_{[0,1]}d\right\}E_y(0).
\end{equation}

Moreover, since $\lambda>0$, by Theorem \ref{thmref} we have
\begin{equation}\label{wer3}
\begin{aligned}
&\left(1- \frac{K_a}{2}-M\right)\int_{Q_T}\left(\eta y_x^2 +\frac{y_t^2}{\sigma} \right)dxdt \\&+\lambda \int_{Q_T} \frac{y^2}{\sigma d}\left(1+\frac{K_a}{2}-x\left(\frac{d'}{d}+\frac{a'-b}{a}\right)\right)dxdt\\
=&\left(1- \frac{K_a}{2}-M\right)\int_{Q_T}\left(\eta y_x^2 +\frac{y_t^2}{\sigma} -\lambda \frac{y^2}{\sigma d}\right)dxdt \\
&+\lambda \int_{Q_T} \frac{y^2}{\sigma d}\left(\left(1- \frac{K_a}{2}-M\right)+1+\frac{K_a}{2}-x\left(\frac{d'}{d}+\frac{a'-b}{a}\right)\right)dxdt\\
\geq & 2T\left(1- \frac{K_a}{2}-M\right)E_y(0)+\lambda (2-K_a-K_d-M)\int_{Q_T} \frac{y^2}{\sigma d}dxdt.
\end{aligned}
\end{equation}

Hence, from \eqref{puntpar},  \eqref{wer}, \eqref{wer2} and \eqref{wer3} we find
the desired inequality.

\noindent{{\sl Case  $\lambda\leq 0$}}. In this case, in \eqref{bordo1} and \eqref{bordo2}  we can estimate the sum of the two integrals in right-hand-side with $E_y(\tau)=E_y(0)$ (by Theorem \ref{thmref}).

Moreover, as in \eqref{wer3} we find
\[
\begin{aligned}
&\left(1- \frac{K_a}{2}-M\right)\int_{Q_T}\left(\eta y_x^2 +\frac{y_t^2}{\sigma} \right)dxdt\\& +\lambda \int_{Q_T} \frac{y^2}{\sigma d}\left(1+\frac{K_a}{2}-x\left(\frac{d'}{d}+\frac{a'-b}{a}\right)\right)dxdt\\
\geq & 2T\left(1- \frac{K_a}{2}-M\right)E_y(0)+\lambda (2+K_a+K_d)\int_{Q_T} \frac{y^2}{\sigma d}dxdt\\
\geq & 2T\left(1- \frac{K_a}{2}-M\right)E_y(0)-|\lambda| 2(2+K_a+K_d)C_{HP}TE_y(0),
\end{aligned}
\]
and as before the conclusion follows.

Now, let $y$ be the mild solution associated to the initial data $(y_0, y_1) \in \mathcal H_0$. By approximation, as in the proof of Theorem \ref{thmuguaglianza}, we get the same estimates for mild solutions.
\end{proof}

Now assume the following
\begin{hypothesis}\label{hyp4}
Hypothesis \ref{hyp2bis} holds with $1<K_a<2-2M$ and $\dfrac{xb}{a} \in L^\infty(0,1)$.
\end{hypothesis}
Set
\begin{equation}\label{Minfty}
M_\infty:=\left\|\frac{xb}{a}\right\|_{L^\infty(0,1)}.
\end{equation}
Under Hypothesis \ref{hyp4}, Theorems \ref{thmuguaglianza2} and \ref{stima1} still hold. Indeed, all the calculations made in the proofs can be reformulated verbatim, substituting $M$ with $M_\infty$. We briefly collect this fact in the following

\begin{thm}\label{stima11}
Assume hypothesis $\ref{hyp4}$; then the statement of Theorem $\ref{stima1}$ holds with $M$ replaced by $M_\infty$.
\end{thm}

\section{Boundary observability and null controllability}\label{section4}

Following \cite{alabau}, we recall the next definition:
\begin{definition}
Problem \eqref{homogequation} is said to be {\it observable in time $T>0$} via the normal derivative at $x=1$ if there exists a constant $C>0$ such that for any $(y_T^0,y_T^1) \in \mathcal H_0$ the mild solution $y$ of \eqref{homogequation} satisfies
\begin{equation}\label{observable}
C E_y(0) \le \int_0^T y_x^2(t,1)dt.
\end{equation}
Any constant $C$ satisfying \eqref{observable} is called {\it observability constant} for \eqref{homogequation} in time $T$. 
\end{definition}
Setting 
\[
C_T := \sup \{C >0: C  \text{ satisfies } \eqref{observable} \},
\]
we have that \eqref{homogequation} is observable if and only if
\[
C_T = \inf_{(y_T^0,y_T^1) \neq (0,0)} \frac{\int_0^T y_x^2(t,1)dt}{E_y(0)} >0.
\]
The inverse of $C_T$, $c_t:= \ds \frac{1}{C_T}$, is called {\it the cost of observability} (or {\it the cost of control}) in time $T$.

Theorems \ref{stima1} and \ref{stima11} admit the following straightforward corollaries.
\begin{corollary}\label{Observability0}
Assume Hypothesis $\ref{hyp3}$ and
\[
-\frac{2-K_a-2M}{2C_{HP}\left(2+K_a+K_d\right)}<\lambda<0.
\]
If
\begin{equation}\label{cor0}
T>\frac{C_3}{C_4},
\end{equation}
then \eqref{homogequation} is observable in time $T$. Moreover
\[
C_T\geq \frac{C_4T-C_3}{ \eta(1) },
\]
where
\[
C_3:=4 \max\left\{1, \frac{1}{a(1)}\right\}+K_a\max \left\{1, C_{HP}\max_{[0,1]}d\right\}
\]
and 
\[
C_4:=2\left(1- \frac{K_a}{2}-M -|\lambda| C_{HP}(2+K_a+K_d)\right).
\]
 \end{corollary}

\begin{corollary}\label{Observability01}
Assume Hypothesis $\ref{hyp3}$, $\lambda \geq 0$ and $K_a + K_d\le 2-M$.
If
\begin{equation}\label{cor01}
T>\frac{C_5}{C_6},
\end{equation}
then \eqref{homogequation} is observable in time $T$. Moreover
\[
C_T\geq \frac{C_6T-C_5}{ \eta(1) },
\]
where
\[
C_5:=\frac{4}{\epsilon}\max\left\{\frac{1}{a(1)}, 1\right\}+K_a\max \left\{1, C_{HP} \max_{[0,1]}d\right\}
\]
and
\[C_6:=1- \frac{K_a}{2}-M.
\]
\end{corollary}

\begin{remark}
Corollaries \ref{Observability0} and \ref{Observability01} still hold true if we replace Hypothesis \ref{hyp3} with Hypothesis \ref{hyp4} and  $M$ with $M_\infty$. 
\end{remark}

We underline that boundary observability is no longer true when $K_a \ge 2$. Indeed if we consider the pure degenerate case, i.e. $\lambda =0$ and $b \equiv0$, one can consider the same examples given in \cite{BFM2022} with $a(x)=x^K$, $K \ge2$, which shows the failure of boundary observability.
\vspace{0.5cm}

In the rest of this section we will devote to prove null controllability for \eqref{mainequation} stated in Theorem \ref{thmNC}. More precisely, given $\left(u_0, u_1\right) \in L_{\frac{1}{\sigma}}^2(0,1) \times H_{\frac{1}{\sigma}}^{-1}(0,1)$, we look for a control $f \in L^2(0, T)$ such that the solution of \eqref{mainequation} satisfies
\begin{equation}\label{NC}
u(T, x)=u_t(T, x)=0, \quad \text { for all } x \in(0,1) .
\end{equation}

For this last purpose, we first give the definition of a solution for \eqref{mainequation} by transposition, which permits low regularity on the notion of solution itself: such a definition is formally obtained by re-writing the equation of  \eqref{mainequation} as $\ds u_{tt}-\sigma\left(\eta u_x\right)_x -\frac{\lambda}{d}u$ (thanks to \eqref{defsigma}), multiplying by $\ds \frac{v}{\sigma}$ and integrating by parts. Precisely:

\begin{definition}
Let $f \in L_{l o c}^2[0,+\infty)$ and $\left(u_0, u_1\right) \in L_{\frac{1}{\sigma}}^2(0,1) \times H_{\frac{1}{\sigma}}^{-1}(0,1)$. We say that $u$ is a solution by transposition of \eqref{mainequation} if
$$
u \in C^1\left([0,+\infty) ; H_{\frac{1}{\sigma}}^{-1}(0,1)\right) \cap C\left([0,+\infty) ; L_{\frac{1}{\sigma}}^2(0,1)\right)
$$
and for all $T>0$
\begin{equation}\label{transposition}
\begin{aligned}
\left\langle u_t(T), v_T^0\right\rangle_{H_{\frac{1}{\sigma}}^{-1}(0,1), H_{\frac{1}{\sigma}}^1(0,1)} & -\int_0^1 \frac{1}{\sigma} u(T) v_T^1 d x=\left\langle u_1, v(0)\right\rangle_{H_{\frac{1}{\sigma}}^{-1}(0,1), H_{\frac{1}{\sigma}}^1(0,1)} \\
& -\int_0^1 \frac{1}{\sigma} u_0 v_t(0, x) d x+\eta(1) \int_0^T f(t) v_x(t, 1) d t
\end{aligned}
\end{equation}
for all $\left(v_T^0, v_T^1\right) \in H_{\frac{1}{\sigma}}^1(0,1) \times L_{\frac{1}{\sigma}}^2(0,1)$, where $v$ solves the backward problem
\begin{align}\label{backward pblm}
\begin{cases}\ds v_{t t}-a v_{x x}-\dfrac{\lambda}{d}v-b v_x=0, & (t, x) \in(0,+\infty) \times(0,1), \\ v(t, 1)=v(t, 0)=0, & t \in(0,+\infty) \\ v(T, x)=v_T^0(x), & x \in(0,1), \\ v_t(T, x)=v_T^1(x), & x \in(0,1) .\end{cases}
\end{align}
\end{definition}
Observe that, by Theorem \ref{thmmildclassolution}, there exists  a unique mild solution  of \eqref{backward pblm} in $[T, +\infty)$. Now, setting $y(t, x):=v(T-t, x)$, one has that $y$ satisfies \eqref{homogequation} with $y_T^0(x)=v_T^0(x)$ and $y_T^1(x)=-v_T^1(x)$. Hence, thanks to Theorem \ref{thmmildclassolution} one has that there exists a unique mild solution $y$ of \eqref{homogequation} in $[0, +\infty)$.  In particular, there exists a unique mold solution $v$ of \eqref{backward pblm} in $[0,T]$. Thus, we can conclude that there exists a unique mild solution 
$$
v \in C^1\left([0,+\infty) ; L_{\frac{1}{\sigma}}^2(0,1)\right) \cap C\left([0,+\infty) ; H_{\frac{1}{\sigma}}^1(0,1)\right)
$$
of \eqref{backward pblm} in $[0, +\infty)$ 
which depends continuously on the initial data $V_T:=\left(v_T^0, v_T^1\right) \in \mathcal{H}_0$.

By Theorem \ref{thmref}, the energy is preserved in our setting, as well, so that the method of transposition done in \cite{alabau} continues to hold thanks to \eqref{ineqetayx2}, and so there exists a unique solution by transposition $u \in C^1\left([0,+\infty) ; H_{\frac{1}{\sigma}}^{-1}(0,1)\right) \cap C\left([0,+\infty) ; L_{\frac{1}{\sigma}}^2(0,1)\right)$ of \eqref{mainequation}, namely a solution of \eqref{transposition}.

Now, we are ready to pass to null controllability, recalling that by linearity and reversibility of equation \eqref{mainequation}, null controllability for any initial data $\left(u_0, u_1\right)$ is equivalent to exact controllability, see \cite{alabau}. In order to prove that \eqref{mainequation} is null controllable, let us start with
\begin{hypothesis}\label{hyp5}
Assume 
\begin{itemize}
\item Hypothesis \ref{hyp3} with $-\frac{2-K_a-2M}{2C_{HP}\left(2+K_a+K_d\right)}<\lambda<0$ and \eqref{cor0}, or 
\item Hypothesis \ref{hyp3} with $\lambda \geq 0$ and $K_a + K_d\le 2-M$ and \eqref{cor01}, or
\item Hypothesis \ref{hyp4} with $-\frac{2-K_a-2M_\infty}{2C_{HP}\left(2+K_a+K_d\right)}<\lambda<0$ and \eqref{cor0}, or
\item Hypothesis \ref{hyp4} with $\lambda \geq 0$ and $K_a + K_d\le 2-M_\infty$ and \eqref{cor01}.
\end{itemize}
\end{hypothesis}

Moreover, we recall the Hilbert space introduced in \eqref{H0corsivo}
\[
\mathcal H_0 := H^1_{{\frac{1}{\sigma}}}(0,1)\times L^2_{{\frac{1}{\sigma}}}(0,1), 
\]
endowed with the  inner product
\[
\langle (u, v), (\tilde u, \tilde v) \rangle_{\mathcal H_0}:=  \int_0^1\eta u'\tilde u'dx + \int_0^1 v\tilde v\frac{1}{\sigma}dx - \lambda \int_0^1 \frac{u \tilde u}{\sigma d} dx
\]
for every $(u, v), (\tilde u, \tilde v)  \in \mathcal H_0$, which induces the equivalent norm
\[
\|(u,v)\|_{\mathcal H_0}^2:=  \int_0^1 \eta(u')^2dx + \int_0^1 v^2\frac{1}{\sigma}dx - \lambda \int_0^1 \frac{u^2}{\sigma d} dx.
\]
On the product space $\mathcal{H}_0 \times \mathcal{H}_0 $, consider the bilinear form $\Lambda$ defined as
$$
\Lambda\left(V_T, W_T\right):=\eta(1) \int_0^T v_x(t, 1) w_x(t, 1) d t
$$
where  $v$ and $w$ are the solutions of \eqref{backward pblm} associated to the final data $V_T:=\left(v_T^0, v_T^1\right)$ and $W_T:=\left(w_T^0, w_T^1\right)$, respectively. 
The  bilinear form $\Lambda$ is continuous and coercive (see Lemma \ref{cont-coercivity})
and thanks to these properties
 one can prove the null controllability for the original problem \eqref{mainequation}. In particular, defining $T_0$ as the lower bound found in Corollaries \ref{Observability0} and \ref{Observability01}, which changes according to the different assumptions used therein, one can prove the next result, whose proof is postponed in the Appendix.

\begin{thm}\label{thmNC}
Assume Hypothesis $\ref{hyp5}$. Then, for all $T>T_0$ and for every $\left(u_0, u_1\right)$ in $L_{\frac{1}{\sigma}}^2(0,1) \times H_{\frac{1}{\sigma}}^{-1}(0,1)$ there exists a control $f \in L^2(0, T)$ such that the solution of \eqref{mainequation} satisfies \eqref{NC}.
\end{thm}

\section{Appendix}
For the readers' convenience, in this section we prove some results used throughout the previous sections.

\subsection{Proof of Lemma \ref{lemmalimits}}

1. We will actually prove an equivalent fact, usually appearing in integrations by parts, namely that
\[
\lim_{x \rightarrow 0} \eta (x)u(x) y'(x)=0.
\]
For this, consider the function
\[
z(x):=
\eta(x) u(x)  y'(x), \quad x \in (0,1].
\]
As in \cite{BFM2022} or in \cite{fm} one can prove that $z \in W^{1,1}(0,1)$. Indeed
\[
\int_0^1|z|dx \le C \|u\|_{L^2(0,1)}\|y'\|_{L^2(0,1)},
\]
thus $z \in L^1(0,1)$.
Moreover
\[
z'(x)= (\eta y')'u + \eta y'u'
\]
and
\[
\int_0^1\eta|y'u'|dx\le C\|y'\|_{L^2(0,1)}\|u'\|_{L^2(0,1)}.
\]
It remains to prove that $(\eta y')'u\in L^1(0,1)$. To this aim observe that
\[
(\eta y')'u dx= \sigma (\eta y')' \frac{u}{\sigma} dx.
\]
Thus, by H\"older's inequality and Proposition \ref{propL2}, 
\[
\int_0^1|(\eta y')'u| dx\le  \|Ay\|_{\frac{1}{\sigma}}\|u\|_{\frac{1}{\sigma}},
\]
for a positive constant $C$.
Thus $z'$ is summable on $[0,1]$.  Hence $z \in W^{1,1}(0,1) \hookrightarrow C[0,1]$ and there exists
\[
\lim_{x \rightarrow 0}z(x)=\lim_{x \rightarrow 0} \eta(x) u(x)  y'(x)=L \in \R.
\]
We will prove that $L=0$.  If $L \neq 0$ there would exist a  neighborhood $\cal I$ of $0$ such that
\[
\frac{|L|}{2} \le |\eta y'u|,
\]
for all $x \in \cal I$;
but, by  H\"older's inequality,
\[
|u(x)|\le \int_0^x |u'(t) |dt \le  \sqrt{x} \|u'\|_{L^2(0,1)}.
\]
Hence
\[
\frac{|L|}{2} \le |\eta y'u| \le \|\eta\|_{\infty} |y'| \sqrt{x} \|u'\|_{L^2(0,1)}
\]
for all $x \in \cal I$. This would imply that
\[
|y'| \ge \frac{|L|}{2 \|u'\|_{L^2(0,1)} \|\eta\|_{\infty} \sqrt{x} }
\]
in contrast to the fact that $ y' \in L^2(0,1)$.

Hence $L=0$ and the conclusion follows.

2. Using the fact that $u(0)=0$, we have
\[
|u(x)| \le \int_0^x |u'(t)|dt\le \sqrt{x} \|u'\|_{L^2(0,1)};
\]
thus
\[
\frac{x}{a(x)}u^2(x)\le \frac{x^2}{a(x)}\|u'\|_{L^2(0,1)}^2. 
\]
By \eqref{limitxgamma/a}, it follows the thesis.

3. As in the previous point, we know that
\[
|u(x)| \le \sqrt{x} \|u'\|_{L^2(0,1)};
\]
thus
\[
\frac{x}{a(x)d(x)}u^2(x)\le \frac{x^2}{a(x)d(x)}\|u'\|_{L^2(0,1)}^2. 
\]
By assumption we have that 
\[
\frac{1}{a(x)} \le \frac{1}{a(1)x^{K_1}} \quad \text{ and } \quad \frac{1}{d(x)} \le \frac{1}{d(1)x^{K_2}}.
\]
Hence
\[
\frac{x}{a(x)d(x)}u^2(x)\le\frac{1}{a(1)d(1)}\frac{x^2}{x^{K_1+ K_2}}\|u'\|_{L^2(0,1)}^2.
\]
Using the fact that $K_1+ K_2 < 2$, the claim follows immediately.

4.
If $K_1+K_2 <2$, then we can deduce the thesis by the previous point. If $K_1+K_2=2$, then proceeding as before, one has
\[
\frac{x^2}{a(x)d(x)}u^2(x)\le \frac{1}{a(1)d(1)}\frac{x^2}{x^{K_1+ K_2}} u^2(x)=  \frac{1}{a(1)d(1)}u^2(x)
\]
and the thesis follows immediately.

5. It is a straightforward consequence of the first point. Indeed, it is enough to choose any  function  $u\in H^1_{\frac{1}{\sigma}}(0,1)$ such that $u(x)=x$ in a right neighborhood of 0, and squaring.

6. 
 Consider the function $z:=x\eta(x) (u'(x))^2$. We have 
\[
z'=\eta(u')^2+x\eta' (u')^2 + 2x\eta u'u''=\eta(u')^2+2xu'(\eta u')'-x\eta'(u')^2.
\]
Since $K_a \le 1$, one has that $ x\eta'(u')^2 \in L^1(0,1)$. Moreover, $2xu'(\eta u')'$ belong to $L^1(0,1)$; indeed, as before,
\[
x(\eta u')'u' dx= x\sigma (\eta u')' \frac{u'}{\sigma} dx=Au\frac{xu'}{\sigma}.
\]
Clearly,  $\ds Au\frac{xu'}{\sigma}\in L^1(0,1)$;
 hence $z\in W^{1,1}(0,1)$ and $\lim_{x\to 0}z(x)=0.$

7. We can  proceed as above, observing that
\[
\left|x\eta'(u')^2\right|=\left|x\eta \frac{b}{a}(u')^2\right|\leq \left\| \frac{xb}{a}\right\|_{L^\infty(0,1)}\left\| \eta\right\|_{L^\infty(0,1)}(u')^2.
\]

\subsection{Proof of Theorem \ref{thmNC}}
The proof of Theorem \ref{thmNC} is based on the following properties of $\Lambda$.
\begin{lemma}\label{cont-coercivity}
Assume Hypothesis \ref{hyp5}. Then, the bilinear form $\Lambda$ is continuous and coercive.
\end{lemma}
\begin{proof}
By Theorem \ref{thmref}, $E_v$ and $E_w$ are constant in time and thanks to \eqref{ineqetayx2}, one has that $\Lambda$ is continuous. Indeed, by Hölder's inequality and Corollary \ref{equivalenze} ,
$$
\begin{aligned}
\left|\Lambda\left(V_T, W_T\right)\right| & \leq \eta(1) \int_0^T\left|v_x(t, 1) w_x(t, 1)\right| d t \\
& \leq\left(\eta(1) \int_0^T v_x^2(t, 1) d t\right)^{\frac{1}{2}}\left(\eta(1) \int_0^T w_x^2(t, 1) d t\right)^{\frac{1}{2}} \\
& \leq C E_v^{\frac{1}{2}}(T) E_w^{\frac{1}{2}}(T) \\
& \le C\left(\int_0^1 \frac{\left(v_T^1\right)^2(x)}{\sigma} d x+\int_0^1 \eta v_x^2(T, x) d x - \int_0^1 \dfrac{\lambda}{d} v^2(T, x) d x  \right)^{\frac{1}{2}} \times \\
& \left(\int_0^1 \frac{\left(w_T^1\right)^2(x)}{\sigma} d x+\int_0^1 \eta w_x^2(T, x) d x - \int_0^1 \dfrac{\lambda}{d} w^2(T, x) d x \right)^{\frac{1}{2}} \\
& =C\left\|\left(v(T), v_t(T)\right)\right\|_{\mathcal{H}_0}\left\|\left(w(T), w_t(T)\right)\right\|_{\mathcal{H}_0}=C\left\|V_T\right\|_{\mathcal{H}_0}\left\|W_T\right\|_{\mathcal{H}_0}
\end{aligned}
$$
for a positive constant $C$ independent of $\left(V_T, W_T\right) \in \mathcal{H}_0 \times \mathcal{H}_0$.

In a similar way, one can prove that $\Lambda$ is coercive. Indeed, by Theorem \ref{stima1}, for all $V_T \in \mathcal{H}_0$, one immediately has
$$
\Lambda\left(V_T, V_T\right)=\int_0^T \eta(1) v_x^2(t, 1) d t \geq C_T E_v(0)=C_T E_v(T) \geq C\left\|V_T\right\|_{\mathcal{H}_0},
$$
for a positive constant $C$.
\end{proof}
\begin{proof}[Proof of Theorem $\ref{thmNC}$]
 Consider the continuous linear map $\mathcal{L}: \mathcal{H}_0 \rightarrow \mathbb{R}$ defined as
$$
\mathcal{L}\left(V_T\right):=\int_0^1 \frac{u_0 v_t(0, x)}{\sigma} d x-\left\langle u_1, v(0)\right\rangle_{H_{\frac{1}{\sigma}}^{-1}(0,1), H_{\frac{1}{\sigma}}^1(0,1)},
$$
where $v$ is the solution of \eqref{backward pblm} associated to the final data $V_T:=\left(v_T^0, v_T^1\right) \in \mathcal{H}_0$. Thanks to Lemma \ref{cont-coercivity} and by the Lax-Milgram Theorem, there exists a unique $\bar{V}_T \in$ $\mathcal{H}_0$ such that
\begin{equation}\label{L-Mpblm}
\Lambda\left(\bar{V}_T, W_T\right)=\mathcal{L}\left(W_T\right)\, \text{ for all } W_T \in \mathcal{H}_0
\end{equation}
Set $f(t):=\bar{v}_x(t, 1), v$ being the solution of \eqref{backward pblm} with initial data $\bar{V}_T$. Then, by \eqref{L-Mpblm}
\begin{equation}\label{equality1}
\begin{aligned}
\eta(1) \int_0^T f(t) w_x(t, 1) d t & =\eta(1) \int_0^T v_x(t, 1) w_x(t, 1) d t=\Lambda\left(V_T, W_T\right)=\mathcal{L}\left(W_T\right) \\
& =-\left\langle u_1, w(0)\right\rangle_{H_{\frac{1}{\sigma}}^{-1}(0,1), H_{\frac{1}{\sigma}}^1(0,1)}+\int_0^1 \frac{1}{\sigma} u_0 w_t(0, x) d x
\end{aligned}
\end{equation}
for all $W_T \in \mathcal{H}_0$.
Finally, denote by $u$ the solution by transposition of \eqref{mainequation}
\begin{equation}\label{equality2}
\begin{aligned}
\eta(1) \int_0^T f(t) w_x(t, 1) d t & =\left\langle u_t(T), w_T^0\right\rangle_{H_{\frac{1}{\sigma}}^{-1}(0,1), H_{\frac{1}{\sigma}}^1(0,1)}-\int_0^1 \frac{1}{\sigma} u(T) w_T^1 d x \\
& -\left\langle u_1, w(0)\right\rangle_{H_{\frac{1}{\sigma}}^{-1}(0,1), H_{\frac{1}{\sigma}}^1(0,1)}+\int_0^1 \frac{1}{\sigma} u_0 w_t(0, x) d x
\end{aligned}
\end{equation}
By \eqref{equality1} and \eqref{equality2}, it follows that
$$
\left\langle u_t(T), w_T^0\right\rangle_{H_{\frac{1}{\sigma}}^{-1}(0,1), H_{\frac{1}{\sigma}}^1(0,1)}-\int_0^1 \frac{1}{\sigma} u(T) w_T^1 d x=0
$$
for all $\left(w_T^0, w_T^1\right) \in \mathcal{H}_0$. Hence, we have
$$
u(T, x)=u_t(T, x)=0 \quad \text { for all } x \in(0,1) .
$$
\end{proof}

\section*{Statements and Declarations}
The authors declare no competing interests.

\section*{Acknowledgments.} G. F. and D. M. are members of {\it Gruppo Nazionale
 per l'Analisi Matematica, la Probabilità e le loro Applicazioni (GNAMPA)} of Istituto Nazionale di Alta Matematica (INdAM) ”Francesco Severi”. G.F. is a member of UMI group {\it Modellistica
Socio-Epidemiologica (MSE)}, she is supported by {\it FFABR Fondo per il finanziamento
delle attività base di ricerca 2017}, by the GNAMPA project 2023 {\it Modelli differenziali per l'evoluzione del clima e i suoi impatti} (CUP E53C22001930001),
by the GNAMPA project 2024 {\it ``Analysis, control and inverse problems for evolution
equations arising in climate science"} (CUP E53C23001670001) and by the PRIN 2022 PNRR P20225SP98 project.

 D. M. is
supported by the INdAM-GNAMPA Project 2023 {\it Variational and non-variational problems with lack of compactness}, 
CUP E53C22001930001, by the INdAM-GNAMPA Project 2024 {\it  Nonlinear
problems in local and nonlocal settings with applications}, CUP E53C23001670001, and by
{\it  FFABR Fondo per il finanziamento delle attività base di ricerca 2017}.

This work was started when A.S. was visiting the University of Tuscia. A.S. thanks the MAECI (Ministry of Foreign Affairs and International Cooperation, Italy) for funding that greatly facilitated scientific collaborations between Hassan First University of Settat
and University of Tuscia.

\end{document}